\documentclass[preprint,11pt]{elsarticle}%
\usepackage{amsmath}
\usepackage{amsfonts}
\usepackage{amssymb}
\usepackage{graphicx}
\usepackage{amssymb}
\usepackage{lineno}%
\providecommand{\U}[1]{\protect\rule{.1in}{.1in}}
\newtheorem{lemma}{Lemma}
\newtheorem{theorem}{Theorem}

\newenvironment{proof}[1][Proof]{\noindent\textbf{#1.} }{\ \rule{0.5em}{0.5em}}
\journal{Electronic Journal of Linear Algebra}
\begin{document}
\begin{frontmatter}
\author{Stanis\l aw Bia\l as\fnref{lab1}}
\author{Micha\l \ G\'ora\fnref{lab2}}
\ead{gora@agh.edu.pl}
\title{On the existence of Hurwitz polynomials with no Hadamard factorization}
\address[lab1]{The School of Banking and Management, ul. Armii Krajowej 4, 30-150 Krak\'ow, Poland}
\address[lab2]{AGH University of Science and Technology, Faculty of Applied Mathematics,\\ al. Mickiewicza 30, 30-059 Krak\'{o}w, Poland}
\begin{abstract}
A Hurwitz stable polynomial of degree $n\geq1$ has a Hadamard factorization if
it is a Hadamard product (i.e. element-wise
multiplication) of two Hurwitz stable polynomials of degree $n$. It is known that Hurwitz stable polynomials of degrees less than four have a Hadamard factorization. We
show that for arbitrary $n\geq4$ there exists a Hurwitz stable polynomial of
degree $n$ which does not have a Hadamard factorization.
\end{abstract}
\end{frontmatter}

\section{Introduction}

A polynomial is said to be (Hurwitz) stable if all its zeros lie in the open
left-half of the complex plane. In the entire class of polynomials, stable
polynomials occupy a special place and stability of polynomials is a classical
topic having both theoretical and applied significance.

In this short note we deal with the Hadamard factorization problem. Recall,
that a Hurwitz stable polynomial of degree $n\geq1$ has (or\textit{\ admits})
a Hadamard factorization if it is a Hadamard product (i.e. element-wise
multiplication) of two Hurwitz stable polynomials of degree $n$. A problem of
the existence of a Hadamard factorization for a given stable polynomial has
been taken by many authors. It is obvious that $\mathcal{F}_{1}^{+}%
=\mathcal{H}_{1}^{+}$ and $\mathcal{F}_{2}^{+}=\mathcal{H}_{2}^{+}$, where
$\mathcal{H}_{n}^{+}$ denotes the entire family of Hurwitz stable polynomials
of degree $n$ with positive coefficients and $\mathcal{F}_{n}^{+}$ is its
subset containing only polynomials admitting a Hadamard factorization.
Moreover, as follows from Garloff and Shrinivasan \cite{GarShr}, we know that
$\mathcal{F}_{3}^{+}=\mathcal{H}_{3}^{+}$ but $\mathcal{F}_{4}^{+}%
\neq\mathcal{H}_{4}^{+}$. Some necessary and necessary and sufficient
conditions for the existence of a Hadamard factorization for a polynomial of
arbitrary degree can be found in Loredo--Villalobos and Aguirre--Hern\'{a}ndez
\cite{LorVil2, LorVil} (unfortunately, their conditions cannot be effectively
applied in practice) and some topological properties of the set $\mathcal{F}%
_{n}^{+}$ (as openness, non-convexity and arc-connectedness) in
Aguirre--Hern\'{a}ndez \textit{et al.} \cite{Aguetal}.

Note also, that there are some issues in which polynomials admitting a
Hadamard factorization play an important role. For example, in \cite{BiaGor}
the authors of this paper have considered the generalized Hadamard product of
polynomials and have shown, among others, that if $f$ is stable and $g$ has a
Hadamard factorization then the generalized Hadamard product of $f$ and $g$ is
stable (in \cite{BiaGor} one can also find a very simple sufficient condition
for the Hadamard factorizability of a polynomial). Taking into account all
these results it seems to be quite surprising that we still do not know if
there exists a polynomial of an arbitrary degree greater than four which does
not admit a Hadamard factorization. We will show that this is the case.

\section{Preliminary results}

In this part, we introduce the basic notation and remind some results which
will be used in the sequel.

\subsection{Basic notations}

We use standard notation: $\mathbb{N}$ and $\mathbb{R}$ stand for the set of
positive integers and real numbers, respectively; $\mathbb{R}^{n\times n}$
stands for the set of real matrices of order $n\times n$; $\mathbb{\pi}%
_{n}^{+}$ denotes the family of $n$--th degree polynomials with positive coefficients.

\subsection{Stable polynomials}

A polynomial $f\in\mathbb{\pi}_{n}^{+}$ $\left(  n\geq1\right)  ,$%
\begin{equation}
f\left(  s\right)  =a_{0}+a_{1}s+\ldots+a_{n-1}s^{n-1}+a_{n}s^{n},
\label{w1.1}%
\end{equation}
is \textit{Hurwitz stable }(or shortly \textit{stable}) if all its zeros have
negative real parts. It is well known (and easy verified) that a necessary
condition for the stability of a real polynomial is that its coefficients are
all of the same sign; without losing generality we will assume in the sequel
that they are positive.

Let $H_{f}\in\mathbb{R}^{n\times n}$ be the Hurwitz matrix associated with
polynomial (\ref{w1.1}), i.e.%
\begin{equation}
H_{f}=%
\begin{pmatrix}
a_{n-1} & a_{n} & 0 & 0 & \ldots & 0\\
a_{n-3} & a_{n-2} & a_{n-1} & a_{n} & \ldots & 0\\
a_{n-5} & a_{n-4} & a_{n-3} & a_{n-2} & \ldots & 0\\
\vdots & \vdots & a_{n-5} & a_{n-4} & \ldots & 0\\
\vdots & \vdots & \vdots & \vdots & \ddots & \vdots\\
0 & 0 & 0 & 0 & \ldots & a_{0}%
\end{pmatrix}
\text{.} \label{w1.2}%
\end{equation}
From among many interesting properties of the Hurwitz matrix we recall the
following one given by Kemperman \cite{Kem} (see Theorem 2 therein).

\begin{theorem}
\label{th1}If $f\in\mathcal{H}_{n}^{+}$ then every square submatrix of the
Hurwitz matrix $H_{f}$ has positive determinant if and only if all its
diagonal elements are positive.
\end{theorem}

Note that as long as $n\geq4$, the matrix%
\begin{equation}
\left(
\begin{array}
[c]{ccc}%
a_{3} & a_{4} & 0\\
a_{1} & a_{2} & a_{n-1}\\
0 & a_{0} & a_{n-3}%
\end{array}
\right)  \label{w1.3}%
\end{equation}
is a submatrix of matrix (\ref{w1.2}). Thus, in view of Theorem \ref{th1}, it
allows us to conclude that if polynomial (\ref{w1.1}) of degree $n\geq4$ is
stable then%
\begin{equation}
a_{3}a_{2}-a_{1}a_{4}>0,\text{\qquad}a_{n-3}a_{2}-a_{n-1}a_{0}>0 \label{w1.4}%
\end{equation}
and
\begin{equation}
a_{n-3}a_{3}a_{2}-a_{n-3}a_{4}a_{1}-a_{n-1}a_{3}a_{0}>0\text{.} \label{w1.5}%
\end{equation}
These inequalities will be crucial in our further considerations.

\subsection{Polynomials admitting a Hadamard factorization}

Together with polynomial $f$ of the form (\ref{w1.1}) we will consider a
polynomial $g\in\mathbb{\pi}_{n}^{+}$,%
\begin{equation}
g\left(  s\right)  =b_{0}+b_{1}s+\ldots+b_{n-1}s^{n-1}+b_{n}s^{n}\text{,}
\label{w1.6}%
\end{equation}
their Hadamard product $f\circ g\in\mathbb{\pi}_{n}^{+}$ defined as an
element-wise multiplication, i.e.%
\[
\left(  f\circ g\right)  \left(  s\right)  =a_{0}b_{0}+a_{1}b_{1}%
s+\ldots+a_{n-1}b_{n-1}s^{n-1}+a_{n}b_{n}s^{n}%
\]
and their Hadamard quotient $f\diamond g\in\mathbb{\pi}_{n}^{+}$ defined as an
element-wise division, i.e.%
\[
\left(  f\diamond g\right)  \left(  s\right)  =\frac{a_{0}}{b_{0}}+\frac
{a_{1}}{b_{1}}s+\ldots+\frac{a_{n-1}}{b_{n-1}}s^{n-1}+\frac{a_{n}}{b_{n}}%
s^{n}\text{.}%
\]
Following Garloff and Wagner \cite{GarWag}, we say that the polynomial
$f\in\mathcal{H}_{n}^{+}$ has a Hadamard factorization if there exist two
polynomials $f_{1},f_{2}\in\mathcal{H}_{n}^{+}$ such that $f=f_{1}\circ f_{2}%
$. Equivalently, the polynomial $f\in\mathcal{H}_{n}^{+}$ has a Hadamard
factorization if there exists a polynomial $g\in\mathcal{H}_{n}^{+}$ such that
$f\diamond g\in\mathcal{H}_{n}^{+}$.

\section{Main result}

Let $f\in\mathbb{\pi}_{n}^{+}$ be as in (\ref{w1.1}) and let $\delta
_{1}\left(  f\right)  $ and $\delta_{2}\left(  f\right)  $ be two positive
numbers given by%
\begin{equation}
\delta_{1}\left(  f\right)  =\frac{a_{1}a_{4}}{a_{3}a_{2}},\text{\qquad}%
\delta_{2}\left(  f\right)  =\frac{a_{n-1}a_{0}}{a_{n-3}a_{2}}\text{.}
\label{w2.1}%
\end{equation}
We start with two simple observations.

\begin{lemma}
\label{lem1}For every $n\geq4$ and for the polynomial $f\in\mathcal{H}_{n}%
^{+}$ we have%
\begin{equation}
\delta_{1}\left(  f\right)  <1,\text{\qquad}\delta_{2}\left(  f\right)  <1
\label{w2.2}%
\end{equation}
and%
\begin{equation}
\delta_{1}\left(  f\right)  +\delta_{2}\left(  f\right)  <1. \label{w2.3}%
\end{equation}

\end{lemma}

\begin{proof}
The above inequalities follow from the stability of $f$:
conditions~(\ref{w2.2}) and (\ref{w2.3}) are equivalent to (\ref{w1.4}) and
(\ref{w1.5}), respectively.
\end{proof}

\begin{lemma}
\label{lem2}If for $f,g\in\mathbb{\pi}_{n}^{+}$ we have $f\diamond
g\in\mathcal{H}_{n}^{+}$, then%
\[
\delta_{1}\left(  f\right)  <\delta_{1}\left(  g\right)  \quad\text{and}%
\quad\delta_{2}\left(  f\right)  <\delta_{2}\left(  g\right)  .
\]

\end{lemma}

\begin{proof}
It is easy to see that%
\[
\delta_{1}\left(  f\diamond g\right)  =\frac{\delta_{1}\left(  f\right)
}{\delta_{1}\left(  g\right)  }\text{\quad and\quad}\delta_{2}\left(
f\diamond g\right)  =\frac{\delta_{2}\left(  f\right)  }{\delta_{2}\left(
g\right)  }\text{.}%
\]
Then, the thesis is a simple consequence of Lemma \ref{lem1}.
\end{proof}

We are now ready to prove the main result of this work.

\begin{theorem}
For every $n\geq4$ there exists a stable polynomial of degree $n$ that does
not have a Hadamard factorization.
\end{theorem}

\begin{proof}
Fix any $n\geq4$ and suppose, by contradiction, that every stable polynomial
of degree $n$ admits a Hadamard factorization. It implies that for an
arbitrary $g_{0}\in\mathcal{H}_{n}^{+}$, there exists $g_{1}\in\mathcal{H}%
_{n}^{+}$ such that $g_{0}\diamond g_{1}\in\mathcal{H}_{n}^{+}$. Repeating
this reasoning one can show that for $g_{i}\in\mathcal{H}_{n}^{+}$ there
exists $g_{i+1}\in\mathcal{H}_{n}^{+}$ such that $g_{i}\diamond g_{i+1}%
\in\mathcal{H}_{n}^{+}$. In this way we obtain a sequence of stable
polynomials $\left\{  g_{i}\right\}  $ which generates two sequences of
positive numbers $\left\{  \delta_{1}\left(  g_{i}\right)  \right\}  $ and
$\left\{  \delta_{2}\left(  g_{i}\right)  \right\}  $ (as in (\ref{w2.1})). It
follows from Lemma \ref{lem2} that sequences $\left\{  \delta_{1}\left(
g_{i}\right)  \right\}  $ and $\left\{  \delta_{2}\left(  g_{i}\right)
\right\}  $ are increasing. Moreover, by Lemma \ref{lem1}, they are bounded
from above and thus convergent. They are also bounded from below by zero and
thus their limits are positive. It means that%
\begin{equation}
\frac{\delta_{1}\left(  g_{i}\right)  }{\delta_{1}\left(  g_{i+1}\right)
}+\frac{\delta_{2}\left(  g_{i}\right)  }{\delta_{2}\left(  g_{i+1}\right)
}\rightarrow2\label{w2.4}%
\end{equation}
as $i\rightarrow\infty$. On the other hand, according to Lemma \ref{lem1} it
follows from the stability of $g_{i}\diamond g_{i+1}$ that
\[
\frac{\delta_{1}\left(  g_{i}\right)  }{\delta_{1}\left(  g_{i+1}\right)
}+\frac{\delta_{2}\left(  g_{i}\right)  }{\delta_{2}\left(  g_{i+1}\right)
}=\delta_{1}\left(  g_{i}\diamond g_{i+1}\right)  +\delta_{2}\left(
g_{i}\diamond g_{i+1}\right)  <1\text{,}%
\]
for every $i\in\mathbb{N}$. This contradicts to (\ref{w2.4}) completing the proof.
\end{proof}

Garloff and Shrinivasan proved in \cite{GarShr} that there exists a polynomial
of degree $4$ which is not a Hadamard product of stable polynomials of degree
$4$. The above theorem shows that similar examples can be constructed for
every $n\geq4$ (our proof is, unfortunately, not constructive).

\section*{Acknowledgments}

The research work of the second author was partially supported by the Faculty
of Applied Mathematics AGH UST statutory tasks (grant no. 11.11.420.004)
within subsidy of Ministry of Science and Higher Education.

\end{document}